\documentclass[10pt]{amsart}

\usepackage[T1]{fontenc}
\usepackage[small,it]{caption}
\usepackage[english]{babel}
\usepackage{url}
\usepackage{fullpage}
\usepackage{graphicx}                         
\usepackage{amsmath}
\usepackage{amsthm}                      
\usepackage{amssymb} 
\usepackage{mathrsfs}
\usepackage{stmaryrd}
\usepackage{enumerate}
\usepackage{enumitem}
\setlist{nolistsep}
\usepackage{tikz}
\usetikzlibrary{calc}
\usepackage{todonotes}
\usepackage{needspace}

\usepackage{tikz-cd}
\usetikzlibrary{decorations.pathmorphing}

\usepackage{euscript}
\renewcommand{\mathcal}{\EuScript}

\swapnumbers
\theoremstyle{plain}
\makeatletter
\def\swappedhead#1#2#3{%
	\thmnumber{\@upn{\the\thm@headfont#2\@ifnotempty{#1}{.~}}}%
	\thmname{#1}%
	\thmnote{ {\the\thm@notefont(#3)}}}
\makeatother

\usepackage{color}
\usepackage{hyperref}
\usepackage[noabbrev,capitalize]{cleveref}
\hypersetup{
	colorlinks,
	linkcolor={red!50!black},
	citecolor={blue!50!black},
	urlcolor={blue!80!black}
}
\usepackage{palatino}
\usepackage{mathtools}

\newtheorem{thm}{Theorem}[section]

\newtheorem{lem}[thm]{Lemma}
\newtheorem{prop}[thm]{Proposition}

\theoremstyle{definition}

\newtheorem{rem}[thm]{Remark}

\newtheorem{para}[thm]{}

\newcommand{\antishriek}{\text{\raisebox{\depth}{\textexclamdown}}}
\newcommand{\Q}{\mathsf{Q}}
\renewcommand{\P}{\mathsf{P}}
\newcommand{\K}{\mathbf{K}}
\newcommand{\g}{\mathfrak{g}}
\newcommand{\h}{\mathfrak{h}}
\newcommand{\U}{\mathscr{U}}
\newcommand{\Uc}{\mathscr{U}^c}
\newcommand{\cdga}{\mathsf{CDGA}}
\newcommand{\dga}{\mathsf{DGA}}
\newcommand{\ucdga}{\mathsf{CDGA}^u}
\newcommand{\udga}{\mathsf{DGA}^u}
\newcommand{\dgla}{\mathsf{DGLA}}
\newcommand{\lie}{\mathsf{Lie}}
\newcommand{\com}{\mathsf{Com}}
\newcommand{\ass}{\mathsf{Ass}}
\newcommand{\colie}{\mathsf{coLie}}
\newcommand{\cocom}{\mathsf{coCom}}
\newcommand{\coass}{\mathsf{coAss}}
\newcommand{\PP}{\mathsf{P}}
\newcommand{\QQ}{\mathsf{Q}}
\newcommand{\EE}{\mathsf{E}}
\newcommand{\DD}{\mathsf{D}}

\newcommand{\Coo}{\mathsf{C}_\infty}
\newcommand{\Aoo}{\mathsf{A}_\infty}
\newcommand{\Loo}{\mathsf{L}_\infty}
\newcommand{\cat}{\mathcal{C}}
\newcommand{\ho}[1]{\mathrm{Ho}(#1)}
\newcommand{\map}[1]{\mathrm{map}_{#1}}
\newcommand{\rmap}[1]{\mathbf{R}\mathrm{map}_{#1}}
\newcommand{\MC}{\mathrm{MC}}
\renewcommand{\Bar}{\mathrm{B}}
\newcommand{\Bass}{\mathrm{B}}
\newcommand{\Buass}{\mathrm{B}_\mathsf{uAss}}

\newcommand{\Cobucom}{\Omega_\mathsf{uCom}}
\newcommand{\Cobass}{\Omega}
\newcommand{\Blie}{\mathrm{B}_{\mathsf{Lie}}}
\newcommand{\Coblie}{\Omega_{\mathsf{Lie}}}
\newcommand{\Bcom}{\mathrm{B}_{\mathsf{Com}}}
\newcommand{\Bucom}{\mathrm{B}_{\mathsf{uCom}}}
\newcommand{\Sym}{\mathrm{Sym}}

\newcommand{\cotimes}{~\widehat{\otimes}~}

\title{Commutative homotopical algebra embeds into non-commutative homotopical algebra}

\author{Ricardo Campos}
\address{Ricardo Campos\\Institut de Mathématiques de Toulouse, UMR5219, Université de Toulouse, CNRS, UPS, F-31062 Toulouse Cedex 9, France}
\email{ricardo.campos@math.univ-toulouse.fr}
\author{Dan Petersen}
\address{Dan Petersen\\{Matematiska Institutionen, Stockholms Universitet, 106 91 Stockholm, Sweden}}
\email{dan.petersen@math.su.se}
\author{Daniel Robert-Nicoud}
\address{Daniel Robert-Nicoud}
\email{daniel.robertnicoud@gmail.com}
\author{Felix Wierstra}
\address{Felix Wierstra\\Korteweg-de Vries Institute for Mathematics, University of Amsterdam, Science Park 105-107, 1098 XG Amsterdam, Netherlands}
\email{felix.wierstra@gmail.com}



\thanks{}

\begin{document}

\begin{abstract}
	Over a field of characteristic zero, we show that the forgetful functor from the homotopy category of commutative dg algebras to the homotopy category of dg associative algebras is faithful. In fact, the induced map of derived mapping spaces gives an injection on all homotopy groups at any basepoint. We prove similar results both for unital and non-unital algebras, and also Koszul dually for the universal enveloping algebra functor from dg Lie algebras to dg associative algebras. An important ingredient is a natural model for these derived mapping spaces as Maurer--Cartan spaces of complete filtered dg Lie algebras (or curved Lie algebras, in the unital case). 
\end{abstract}

\maketitle

\section{Commutativity and associativity}

\begin{para}
	The forgetful functor from commutative rings to associative algebras is a canonical example of a \emph{fully faithful functor}: the notion of a homomorphism $A \to A'$ does not depend on whether $A$ and $A'$ are considered as commutative or just associative algebras.
\end{para}

\begin{para}
	In the setting of homotopical algebra, the analogous statement is emphatically false. Namely, when a forgetful functor $C \to D$ is fully faithful, this can be informally understood as saying that the category $C$ consists of objects of $D$ satisfying some \emph{property}, and the forgetful functor merely forgets about this property. Examples could be the forgetful functor from finite sets to all sets, or as just mentioned, the forgetful functor from commutative rings to associative algebras. But in a homotopy-theoretic situation we should not be asserting the \emph{property} that $xy=yx$; rather, we should supply a \emph{particular} homotopy between $xy$ and $yx$ (and then higher homotopies, etc.), and these homotopies are additional \emph{structure} on the algebra. 
\end{para}

\begin{para}
	To be more explicit, let us consider the two categories $\ho{\cdga}$ and $\ho{\dga}$ of non-unital commutative dg algebras, and non-unital dg algebras respectively, localized at the class of quasi-isomorphisms. We always work over a field $\K$ of characteristic zero in what follows. In this case both categories can be modeled by the Hinich model structure \cite{hinichmodelstructure}, in which all objects are fibrant. This means in particular that the set of morphisms $A \to A'$ in the homotopy category can be computed as the set of homotopy classes of morphisms $\mathrm Q A \to A'$, where $\mathrm QA$ denotes a cofibrant replacement of $A$. Now on the one hand a cofibrant object of $\cdga$ very rarely remains cofibrant in $\dga$; on the other hand, the notions of homotopy in the two categories are also different. The consequence is that if $A, A'$ are commutative dg algebras, then morphisms $A\to A'$ in $\ho{\dga}$ are given by an a priori larger set of maps modulo an a priori coarser notion of homotopy than morphisms in $\ho{\cdga}$. We denote similarly by $\ho\ucdga$ and $\ho\udga$ the homotopy categories of unital commutative dg algebras and unital dg algebras, respectively. Everything we have said in this paragraph remains true in the unital case, too. 
\end{para}

\begin{para}\label{lurie example}
	In fact, the functor $\ho\udga \to \ho\udga$ is \emph{not} full. Take for example $A$ to be the polynomial ring $\K[x,y]$, considered as a commutative dg algebra concentrated in degree $0$. It is cofibrant in the category $\ucdga$, but not in $\udga$. A cofibrant replacement  in $\udga$ is given by non-commutative polynomial ring $\K\langle x, y, z\rangle$ with $\vert z \vert = 1$ and $dz = xy-yx$.  Therefore, for any $A'$ in $\ucdga$, we have
	\[
	\hom_{\ho{\ucdga}}(A, A')\cong H_0(A')\times H_0(A'),
	\]
	but
	\[
	\hom_{\ho{\udga}}(A, A')\cong H_0(A')\times H_0(A') \times H_1(A').
	\]
	The same example would work just as well if we had worked with non-unital algebras.
	Our main result is the following.
\end{para}

\begin{thm}\label{mainthm}
	Over a base field of characteristic zero, the forgetful functors
	\[
	\ho{\cdga}\longrightarrow\ho{\dga}
	\qquad \text{and}\qquad \ho{\ucdga}\longrightarrow\ho{\udga}
	\]
	are faithful.
\end{thm}

\begin{para}
	The argument we will use for this result will in fact prove a stronger result at the level of simplicial mapping spaces. We denote by $\rmap{\cat}$ the derived mapping space of a model category $\cat$.
\end{para}

\begin{thm}\label{mainthm2}
	Over a base field of characteristic zero, the natural maps
	\[
	\rmap{\cdga}(A, A')\longrightarrow\rmap{\dga}(A, A')
	\]
	and
	\[
	\rmap{\ucdga}(A, A')\longrightarrow\rmap{\udga}(A, A')
	\]
	induced by the forgetful functors give rise to an injection on all homotopy groups at any basepoint, for any $A, A'\in\cdga$, resp.\ $\ucdga$.
\end{thm}

\begin{para}\label{amrani remark}
	For the homotopy groups $\pi_n$ with $n > 0$ this recovers a theorem of Amrani \cite{amrani}, so only the result about $\pi_0$ (i.e. \cref{mainthm}) is new. The assertion about $\pi_0$ is significantly trickier to prove.
\end{para}

\begin{para}
	Although the structure of the proof is very similar in both the unital and non-unital case, the argument is technically simpler when dealing with not necessarily unital algebras. For the rest of the introduction we focus on this case. \label{focus on non-unital case}
\end{para}

\begin{rem}
	The set of homotopy classes of maps $A \to A'$ in $\ho\dga$ can also be understood in terms of $\Aoo$-morphisms. This has the advantage that one does not have to cofibrantly replace $A$. Namely, $\hom_{\ho\dga}(A,A')$ can be identified with the set of $\infty$-morphisms of $\Aoo$-algebras from $A$ to $A'$, modulo $\Aoo$-homotopy. Similarly $\hom_{\ho\cdga}(A,A')$ is the set of $\Coo$-morphisms from $A$ to $A'$ modulo $\Coo$-homotopy. Thus, \cref{mainthm} can be reformulated as saying that if two $\Coo$-morphisms are $\Aoo$-homotopic, then they are also $\Coo$-homotopic. Such a result seems potentially quite useful since the notion of $\Aoo$-homotopy is much simpler than the notion of $\Coo$-homotopy, the reason being that the $\Aoo$-operad is non-symmetric. Indeed, suppose that $f, g : A \to A'$ are $\Aoo$-morphisms between $\Aoo$-algebras, and let 
	\[
	\Bass f, \Bass g : \Bass A \to \Bass A'
	\]
	be the corresponding maps of coassociative coalgebras. By a homotopy between $f$ and $g$ we mean a $(\Bass f,\Bass g)$-coderivation $H:\Bass A \to \Bass A'$ such that $\partial H = \Bass f - \Bass g$. We refer to \cite[Section 1.3.4]{lefevre-hasegawa} for more details. The notion of homotopy of $\infty$-morphisms over a symmetric operad is more complicated and is discussed in detail in \cite{dotsenko-poncin}. 
\end{rem}

\begin{para}
	This paper may be considered as a sequel, or coda, to our previous paper \cite{CPRNW19}. We say that a functor $F \colon C \to D$ is \emph{injective on isomorphism classes} if $F(X) \cong F(Y)$ implies $X \cong Y$. In Theorem A of  \cite{CPRNW19} we proved the result that $\ho\cdga\to\ho\dga$ is injective on isomorphism classes over a field of characteristic zero (and similarly in the unital case). This result and \cref{mainthm} are morally similar in that both try to probe the failure of the forgetful functor to be fully faithful (as clearly a fully faithful functor is injective on isomorphism classes). Putting the two papers together we see that the homotopy theory of commutative dg algebras embeds faithfully into the non-commutative homotopy theory in a precise sense. What is also striking is that the essential structure of the two arguments is extremely similar. 
\end{para}

\begin{para}
	Let us try to bring out this similarity. The proof of  \cite[Theorem A]{CPRNW19} quickly reduces to the following statement: given two $\Coo$-algebra structures on the same chain complex, and an $\Aoo$-isotopy between them, there must also exist a $\Coo$-isotopy between them. Similarly \cref{mainthm} reduces to the statement that given two $\Coo$-morphisms between $\Coo$-algebras and an $\Aoo$-homotopy between them, there also exists a $\Coo$-homotopy between them. To treat both situations on the same footing, we say simply that we are given two $ \Coo$-objects $\xi$ and $\xi'$ related by an $\Aoo$-gauge, and we want to find a $\Coo$-gauge from $\xi$ to $\xi'$.  More precisely, in both cases the problems are best described in terms of deformation problems governed by dg Lie algebras. In our previous paper we considered two  dg Lie algebras whose Maurer--Cartan elements correspond to $\Aoo$-algebra structures (resp. $\Coo$-algebra structures) on some fixed chain complex, such that gauges correspond to $\Aoo$-isotopies (resp.\ $\Coo$-isotopies). Similarly, in the current paper we have two dg Lie algebras that describe the set of $\Aoo$-morphisms (resp. $\Coo$-morphisms) between two fixed such algebras, with gauges now corresponding to $\Aoo$-homotopies (resp. $\Coo$-homotopies). In both cases there is an inclusion of the Lie algebra encoding $\Coo$-structures into the Lie algebra encoding $\Aoo$-structures. 
\end{para}

\begin{para}\label{outline}
	Once the problem is formulated in terms of $\Aoo$- and $\Coo$-structures, it is natural to try to work one coefficient at a time, which amounts to noting that the dg Lie algebras in question come with natural filtrations (or gradings) by arity of the $\Aoo$- and $\Coo$-operations. Let $\tau_{\leq n}\xi$ be the truncation of $\xi$ to a $\mathsf C_n$-object, discarding all operations of arity bigger than $n$. By an inductive obstruction theory argument one can in both cases show that $\tau_{\leq n}\xi$ and $\tau_{\leq n}\xi'$ are $\mathsf C_n$-gauge equivalent for every $n$. Indeed, if we have an equivalence between $\tau_{\leq n}\xi$ and $\tau_{\leq n}\xi'$, then there is a well defined obstruction class for finding a gauge equivalence between $\tau_{\leq (n+1)}\xi$ and $\tau_{\leq (n+1)}\xi'$. More specifically, there are \emph{two} classes $\mathrm{o}_{n+1}^\mathsf{A} \in H^\mathsf{A}$ and $\mathrm{o}_{n+1}^\mathsf{C} \in H^\mathsf{C}$, which are obstructions for finding an $\mathsf A_{n+1}$-gauge and a $\mathsf C_{n+1}$-gauge, respectively, both lying in certain cohomology groups. Now the great news in both situations is that $H^\mathsf{C}$ is a direct summand of $H^\mathsf{A}$, and $\mathrm{o}_{n+1}^\mathsf{C}$ is the projection of $\mathrm{o}_{n+1}^\mathsf{A}$ onto $H^\mathsf{C}$. But $\mathrm{o}_{n+1}^\mathsf{A}$ must vanish, since we had an $\Aoo$-gauge from $\xi$ to $\xi'$ by assumption, and then the obstruction $\mathrm{o}_{n+1}^\mathsf{C}$ vanishes, too. (In \cite{CPRNW19} the groups $H^\mathsf{C}$ and $H^\mathsf{A}$ are Harrison and Hochschild cohomology groups, respectively, and Harrison cohomology is a direct summand of Hochschild cohomology by a result of Barr \cite{barrhochschild}.) But at this point one gets stuck: there is no natural way to deduce from the fact that $\tau_{\leq n} \xi$ is $\mathsf{C}_n$-gauge equivalent to $\tau_{\leq n} \xi'$ for all $n$, that $\xi$ and $\xi'$ are $\Coo$-gauge equivalent. Although we can find a gauge at each step, this does not imply that transfinite composition of all these gauges is well defined or convergent. 
\end{para}

\begin{para}
	To make the argument work one needs to be more careful in the construction. Instead of saying at each step that a certain obstruction class vanishes in cohomology, so that \emph{some} gauge exists, we need to be able to choose a \emph{specific} gauge, which is small enough in a suitable sense, so that the sequence of gauges converges in the limit. To accomplish this, we pass from cohomology to cochains and consider the natural complete filtrations on everything in sight. The difficulties involved in making this strategy work can be bundled up into a criterion which we proved in \cite{CPRNW19}.
\end{para}

\begin{thm}[{\cite[Theorem 1.7]{CPRNW19}}]\label{criterion}
	Let $i \colon \h \to \g$ be a morphism of complete filtered dg Lie algebras over a field of characteristic zero. Suppose that there exists $r \colon \g \to \h$ such that $r \circ i = \mathrm{id}_\h$, and $r$ is a morphism of filtered $\h$-modules. Then the map $\MC_\bullet(\h) \longrightarrow \MC_\bullet(\g)$ induces an injection on all homotopy groups for all basepoints. In particular, the set of gauge equivalence classes of Maurer--Cartan elements of $\h$ injects into the gauge equivalence classes of Maurer--Cartan elements of $\g$. 
\end{thm}

\begin{para}
	The assumption that $\h \to \g$ admits a retraction as $\h$-modules immediately implies that for any Maurer--Cartan element $\alpha$ of $\h$, there is a split injection in homology $H_\bullet(\h^\alpha) \to H_\bullet(\g^\alpha)$, where $(-)^\alpha$ denotes the effect of twisting the differential by the Maurer--Cartan element $\alpha$. In terms of the heuristic outline of \S\ref{outline}, this corresponds to the fact that $H^\mathsf{C}$ is a direct summand of $H^\mathsf{A}$, and that all obstruction classes vanish. Having the retraction be filtration-preserving is what gives the extra leverage and allows passage to the limit. We can now also justify the remark made in \S\ref{amrani remark}, that the assertion about higher homotopy groups is easier than the assertion about $\pi_0$ in \cref{mainthm2}. Indeed, for $n>0$ a theorem of Berglund \cite{berglund2015rational} furnishes an identification $\pi_n(\MC_\bullet(\h),\alpha) \cong H_{n-1}(\h^\alpha)$, so for the injection on higher homotopy groups it is not necessary to suppose compatibility of the retraction with the filtrations, and the convergence problems are irrelevant.
\end{para}

\begin{para}
	In order to apply \cref{criterion} we need suitable models for mapping spaces in $\cdga$ and $\dga$ in terms of Maurer--Cartan spaces of complete filtered dg Lie algebras. Before stating the result let us recall some notation and preliminaries. If $A$ is a dg algebra, we denote by $\Bass A$ its bar construction. If furthermore $A$ is commutative, we write $\Bcom A$ for its ``commutative'' bar construction, i.e.\ the Koszul dual Lie coalgebra of $A$. Recall also that the space of linear maps from a dg Lie coalgebra to a commutative dg algebra is itself a dg Lie algebra. (This is the ``linear dual'' of the perhaps more familiar fact that the tensor product of a Lie algebra and a commutative ring is a Lie algebra.) Similarly the space of linear maps from a dg coalgebra to a dg algebra is itself a dg algebra, and in particular a dg Lie algebra. 
\end{para}

\begin{prop}\label{mappingspaces}
	Let $A$ and $A'$ be dg algebras over a field $\K$ of characteristic zero. Then:
	\begin{itemize}
			\item[$\diamond$] $\rmap{\dga}(A, A')\simeq\MC_\bullet(\hom_\K(\Bass A, A'))$.
			\item[$\diamond$] If $A$ and $A'$ are commutative, then $\rmap{\cdga}(A, A')\simeq\MC_\bullet(\hom_\K(\Bcom A, A'))$.
		\end{itemize}
	Here, both $\hom_\K(\Bass A, A')$ and $\hom_\K(\Bcom A, A')$ are considered as complete filtered dg Lie algebras, with the filtrations induced by the coradical filtrations of $\Bass A$ and $\Bcom A$ respectively. In the commutative case, the projection $\Bass A\to \Bcom A$ induces the natural map $\rmap{\dga}(A, A')\longrightarrow\rmap{\cdga}(A, A').$
\end{prop}

\begin{rem}
	Versions of \cref{mappingspaces} are already in the literature, see in particular Berglund \cite[Proposition 6.1]{berglund2015rational}. But note the lack of boundedness or finite-dimensionality assumptions here. 
\end{rem}

\begin{rem}
	We expect the first equivalence $\rmap{\dga}(A, A')\simeq\MC_\bullet(\hom_\K(\Bass A, A'))$ to remain valid over fields of arbitrary characteristic, {cf.}\ \cite{DKW18}.
\end{rem}

\begin{para}\label{sketch of proof in intro}
	The proof of \cref{mainthm2} is now nearly complete: by \cref{criterion} and \cref{mappingspaces}, all that is left to do is to construct a certain left inverse to the map
	\[
	\hom_\K\left(\Bcom A, A'\right)\longrightarrow\hom_\K\left(\Bass A, A'\right).
	\]
	We will obtain it as follows. We can identify $\Bass A$ with the \emph{universal enveloping conilpotent coalgebra} of $\Bcom A$. By a dual version of the Poincar\'e--Birkhoff--Witt theorem, we can identify $\Bass A\cong\Sym^c(\Bcom A)$ as $\Bcom A$-comodules. From this, one can in particular construct a splitting of the projection $\Bass A \to \Bcom A$. 
\end{para}

\begin{para}Again we see that the arguments of this paper follow closely those of \cite{CPRNW19}, where it also turned out that the retraction needed to apply \cref{criterion} was obtained from the Poincar\'e--Birkhoff--Witt theorem.
\end{para}

\begin{para}
	As indicated in \S\ref{focus on non-unital case}, there are additional complications in the unital case: namely, when the algebras have units, we need to replace dg Lie algebras with \emph{curved} Lie algebras throughout, as we explain in \cref{sect unital}.
\end{para}

\begin{para}
	In our previous paper \cite{CPRNW19} we did not only study the forgetful functor $\cdga \to \dga$. A major portion of the paper concerned instead the \emph{Koszul dual} of this functor. Since associative algebras are Koszul self-dual, and commutative algebras are Koszul dual to Lie algebras, this Koszul dual is the universal enveloping algebra functor $\dgla \to \dga$. In \cite{CPRNW19} the results we obtained about the universal enveloping algebra functor were slightly weaker than those for the forgetful functor. Indeed, we could prove that $\ho\cdga\to\ho\dga$ is injective on isomorphism classes, but we could only prove that the universal enveloping algebra functor $\ho\dgla\to\ho\dga$ is injective on isomorphism classes up to \emph{derived completion} of the dg Lie algebras involved. By contrast, in this paper we are able to prove a ``Koszul dual'' of \cref{mainthm} with no additional completeness hypotheses. 
\end{para}

\begin{thm}\label{thm:B}
	Let $\g, \h$ be two dg Lie algebras over a field of characteristic zero. For any choice of basepoint, the natural map
	\[
	\rmap{\dgla}(\g, \h)\longrightarrow\rmap{\dga}(U\g, U\h)
	\]
	induces an injection on all homotopy groups.
\end{thm}

\begin{para}
	This theorem is less striking than its commutative and associative counterpart \cref{mainthm} in that there is no longer a large difference in behavior between the classical (non-dg) and the homotopical situation. Indeed, consider the classical universal enveloping algebra functor from Lie algebras to associative algebras over a field of characteristic zero. This functor is clearly not full (consider e.g.\ $\g=\h$ a free Lie algebra). However, the universal enveloping algebra functor \emph{does} turn out to be faithful, but it is not automatic: it is a consequence of the Poincar\'e--Birkhoff--Witt theorem. This is also exactly what happens in the homotopical situation.
\end{para}

\begin{para}\label{par:final of intro}
	Some experts will perhaps not find anything surprising beyond this point: what remains of the paper is mostly to fill in details in the proof of \cref{mainthm2} sketched in \S\ref{sketch of proof in intro}. We first explain the dual form of the PBW theorem mentioned in \S\ref{sketch of proof in intro} in \cref{section: PBW}, and verify that the retraction obtained from dual PBW satisfies the conditions of \cref{criterion} in \cref{section: retraction}. After that we give the proof of \cref{mappingspaces} in \cref{section: mapping spaces}. In \cref{sect unital} we then treat the unital case, proving in particular a unital version of \cref{mappingspaces} involving curved Lie algebras. Finally, in Section \ref{sec:enveloping algebras} we consider the Koszul dual situation which concerns the universal enveloping algebra functor from dg Lie algebras to dg algebras.
\end{para}

\begin{para}
	We mostly follow the same conventions as in \cite{CPRNW19}. We always work over a field $\K$ of characteristic $0$. All the filtrations that we consider are concentrated in positive filtration degree.
\end{para}

\begin{para}
	As in \cite{CPRNW19} we systematically use the language of algebraic operads, and the results are obtained by studying the interplay between the operads $\mathsf{Lie}$, $\mathsf{Ass}$ and $\mathsf{Com}$. In fact, the only property of these operads that we end up using (besides their Koszulness) is that the natural morphism $\mathsf{Lie} \to \mathsf{Ass}$ admits a left inverse in the category of infinitesimal bimodules over the operad $\mathsf{Lie}$. The methods of this paper extend to show the following more general theorem, whose proof we leave to the interested reader:
\end{para}

\begin{thm}\label{generalization}
	Let $f \colon \P \to \Q$ be a morphism of binary Koszul operads in characteristic zero with $\P(n)$ and $\Q(n)$ finite dimensional for all $n$. Let $\Q^! \to \P^!$ be the induced morphism between the Koszul dual operads. Suppose that there exists a morphism of infinitesimal $\P$-bimodules $s \colon \Q \to \P$ such that $s \circ f = \mathrm{id}_\P$. Then both the forgetful functor from $\Q^!$-algebras to $\P^!$-algebras and the derived operadic pushforward functor $\mathbf L f_!$ from $\P$-algebras to $\Q$-algebras have the property that they induce injections on all homotopy groups of derived mapping spaces for all basepoints. 
\end{thm}

\begin{para} \textbf{Acknowledgments.}
	We would like to thank Victor Roca i Lucio for discussions concerning his thesis and Joost Nuiten for discussions about curved Lie algebras. The first author was supported by the grant ANR-20-CE40-0016 HighAGT, the second author was supported by the grant ERC-2017-STG 759082 and a Wallenberg Academy Fellowship, and the fourth author was supported by the Dutch Research Organization (NWO) grant number VI.Veni.202.046.
\end{para}

\section{On the Poincar\'e--Birkhoff--Witt theorem}\label{section: PBW}

\begin{para}
	In this section we will explain some generalities around the Poincar\'e--Birkhoff--Witt theorem. Our perspective will be that the statement of the Poincar\'e--Birkhoff--Witt theorem is that there is an isomorphism of infinitesimal $\lie$-bimodules
	\begin{equation}\label{infbimod}\tag{$\dagger$}
		\ass \cong \com \circ \lie.
	\end{equation}
	We may also write the right hand side as  $\Sym(\lie)$, which makes the infinitesimal $\lie$-bimodule structure more transparent (cf.\ \cite[Remark~2.9]{CPRNW19}). In \cite[Section~2]{CPRNW19} we explained how to deduce \eqref{infbimod} from (an abstract form of) the usual PBW theorem, i.e.\ that for a Lie algebra $\g$ there is an isomorphism of $\g$-modules $U\g \cong \Sym(\g)$. We will now explain how to deduce the standard form of the PBW theorem from \eqref{infbimod}, and then explain that the exact same reasoning also implies the following dual form of the PBW theorem.
\end{para}

\begin{prop}\label{dualpbw}
	Let $L$ be a conilpotent dg Lie coalgebra, and denote by $U^c L$ its universal enveloping conilpotent coalgebra. There is an isomorphism of $L$-comodules $U^c(L) \cong \Sym^c(L)$.
\end{prop}

\begin{para}
	The material in this section is known to experts, but we do not know a reference which states it explicitly, and the perspective is perhaps somewhat novel. The original dual PBW theorem is due to Michaelis \cite{michaelis80, michaelis85}, who considers more generally any \emph{locally finite} Lie coalgebra, not necessarily conilpotent. The fact that the isomorphism \eqref{infbimod} can be promoted to an isomorphism of infinitesimal bimodules seems to have been first noticed by Griffin \cite{griffincomodules}. Other closely related works are Dotsenko--Tamaroff \cite{dotsenkotamaroff} and Fresse \cite[Section 10.2]{fressemodules}.
\end{para}

\begin{para}
	Let $\PP\to\QQ$ be a morphism of operads. The restriction functor from $\QQ$-algebras to $\PP$-algebras has a left adjoint, which we denote $\U_\QQ$. In the case of $\lie\to\ass$, the functor $\U_{\ass}$ is the usual universal enveloping algebra functor.
\end{para}

\begin{para}
	In general, we can identify algebras over an operad $\PP$ with  left $\PP$-modules concentrated in arity $0$. The functor $\U_\QQ$ makes sense for left $\PP$-modules as well, and in a sense it is more naturally considered as such. If $V$ is a left $\PP$-module, then $\U_\QQ V$ is the relative composition product
	\begin{equation}\label{coeq}\tag{$\ast$}
	\U_\QQ V\coloneqq\QQ\circ_\PP V = \mathrm{coeq}(\QQ \circ \PP \circ V \rightrightarrows \QQ \circ V),
	\end{equation}
	i.e.\ the coequalizer of two parallel arrows, one of which is given by the left $\PP$-module structure on $V$ and the other by the right $\PP$-module structure on $\QQ$ induced by the morphism $\PP\to\QQ$.
\end{para}

\begin{para}
	We may now reverse the perspective by fixing $V$ and taking \eqref{coeq} to define a functor of $\QQ$. We immediately notice that \eqref{coeq} makes no reference to the morphism $\PP\to\QQ$; it only involves the right $\PP$-module structure on $\QQ$. For simplicity, from now on we will now fix $V=A$ to be a $\PP$-algebra. Then, in differing levels of generality,\footnote{We remark that the fact that $\PP$-bimodules are neither more nor less general than infinitesimal $\PP$-bimodules is reflected in the fact that $\PP$-algebras are neither more nor less general than $A$-modules.} we may consider \eqref{coeq} as defining either:
	\begin{itemize}
		\item[$\diamond$] a functor $\QQ\longmapsto\U_\QQ A$ from right $\PP$-modules to chain complexes,
		\item[$\diamond$] a functor $\QQ\longmapsto\U_\QQ A$ from $\PP$-bimodules to $\PP$-algebras, or
		\item[$\diamond$] a functor $\QQ\longmapsto\U_\QQ A$ from infinitesimal $\PP$-bimodules to $A$-modules.
	\end{itemize}
	The last bullet point is the least obvious here. An infinitesimal left $\PP$-module structure on $\QQ$ consists of a map
	\[
	\PP\circ(\PP;\QQ)\to\QQ
	\]
	satisfying a suitable associativity condition, with notation as in \cite[Section~6.1]{lodayvallette}. Applying the relative composition product $(-)\circ_\PP A$ to both sides, this yields a map
	\[
	\PP\circ(A;\U_\QQ A)\longrightarrow\U_\QQ A
	\]
	defining an $A$-module structure on $\U_\QQ A$.
\end{para}

\begin{para}\label{pbw deduction}
	Specializing this to $\lie\to \ass$ and using \eqref{infbimod}, we have that for any dg Lie algebra $\g$
	\[
	U\g = \U_{\ass}\g\cong\U_{\com\circ\lie}\g=\com\circ\lie\circ_{\lie}\g\cong\com\circ\g=\Sym(\g)
	\]
	as $\g$-modules, recovering the classical PBW theorem.
\end{para}

\begin{para}\label{comodule functoriality}
	We can now dualize the whole discussion. All cooperads, coalgebras, and comodules are implicitly assumed to be conilpotent in what follows. Suppose that $\EE\to\DD$ is a morphism of cooperads. In this situation, there is a natural restriction functor from $\EE$-coalgebras to $\DD$-coalgebras. Dually to the case of operads, it has a \emph{right} adjoint, which we will denote $\Uc_\EE$. A $\DD$-coalgebra is a left $\DD$-comodule concentrated in arity $0$, and we may define $\Uc_\EE M$ more generally for any left $\DD$-comodule $M$ by the formula
	\[
	\Uc_\EE M\coloneqq \EE\circ^\DD M,
	\]
	where the right-hand side is the equalizer of the diagram
	\[
	\EE\circ M\rightrightarrows\EE\circ\DD\circ M
	\]
	dual to \eqref{coeq}. Again we fix a $\DD$-coalgebra $C$, and consider the functor $\EE \mapsto \Uc_\EE C$. This may be considered as a functor from right $\DD$-comodules to chain complexes, from $\DD$-bicomodules to $\DD$-coalgebras, or from infinitesimal $\DD$-bicomodules to $C$-comodules.
\end{para}

\begin{para}
	In the special case of the morphism $\coass\to\colie$, the functor $\Uc_{\coass} = U^c$ is the universal enveloping conilpotent coalgebra. Taking the linear dual in \eqref{infbimod} shows that there is an isomorphism of infinitesimal $\colie$-comodules
	\begin{equation}\label{infcobimod}
		\coass \cong \cocom \circ \colie.\tag{$\dagger\dagger$}
	\end{equation}
	Dualizing the argument of \S\ref{pbw deduction} in the evident way proves \cref{dualpbw}. \end{para}

\section{Construction of a retraction}\label{section: retraction}

\begin{para}
	We write $\hom_\K$ for the internal hom of chain complexes over $\K$. If $C$ is a dg coalgebra and $A$ is a dg algebra, then $\hom_\K(C, A)$ has a natural dg algebra structure defined by the \emph{convolution product}. If we write the coproduct of $C$ using Sweedler notation
	\[
	\Delta(x) = \sum x_{(1)}\otimes x_{(2)},
	\]
	then the convolution product of $f, g\in\hom_\K(C, A)$ is defined as
	\[
	(f \ast g)(x)\coloneqq\sum f(x_{(1)})g(x_{(2)})
	\]
	for $x\in C$. This is readily verified to be associative. By the same formulas, if $M$ is a right $C$-comodule and $N$ is a right $A$-module, then $\hom_\K(M, N)$ is functorially a right $\hom_K(C, A)$-module.
\end{para}

\begin{para}\label{lie functoriality}
	In a similar way, if $L$ is a dg Lie coalgebra and $A$ is a commutative dg algebra, then $\hom_\K(L, A)$ is a dg Lie algebra. Again using Sweedler's notation for the Lie cobracket of $L$, the bracket of $f, g\in\hom_\K(L, A)$ is given by
	\[
	[f, g](x)\coloneqq\sum f(x_{(1)})g(x_{(2)})
	\]
	for $x\in L$. If $M$ is an $L$-comodule and $N$ an $A$-module, then $\hom_\K(M,N)$ is functorially a $\hom_\K(L,A)$-module.
\end{para}

\begin{para}\label{filtration and conilpotence}
	If $\DD$ is a reduced cooperad, then any $\DD$-coalgebra comes equipped with a canonical filtration as follows. The cooperad $\DD$ has a natural increasing filtration by arity
	\[
	0 = \tau_{\leq0}\DD\xhookrightarrow{\quad}\tau_{\leq1}\DD\xhookrightarrow{\quad}\tau_{\leq2}\DD\xhookrightarrow{\quad}\cdots
	\]
	where $\tau_{\leq k}\DD$ denotes the truncation of $\DD$ at arity $k$, i.e.\ the cooperad obtained by setting all operations of arity above $k$ to zero. Each of the truncations is a $\DD$-cobimodule in a natural way. If $C$ is a $\DD$-coalgebra, then $C\cong\DD\circ^\DD C$ inherits a filtration given by $\tau_{\leq k}\DD\circ^\DD C$ in filtration degree $k-1$, called the \emph{coradical filtration} of $C$. The $\DD$-coalgebra $C$ is said to be \emph{conilpotent} if this filtration is exhaustive. With this description of the coradical filtration, it becomes obvious that if $\EE\to\DD$ is a morphism of reduced cooperads, then the natural map
	\[
	\Uc_\EE C = \EE\circ^\DD C\longrightarrow \DD\circ^\DD C \cong C
	\]
	is compatible with the respective coradical filtrations, where $\Uc_\EE C$ is endowed with the filtration induced by the arity filtration of $\EE$. More generally, if $\EE_1\to\EE_2$ is a morphism of right $\DD$-comodules, $\DD$-bicomodules, or infinitesimal $\DD$-bicomodules, then the induced morphism
	\[
	\Uc_{\EE_1}C\longrightarrow\Uc_{\EE_2}C
	\]
	is compatible with the filtrations induced by the arity filtrations of $\EE_1$ and $\EE_2$.
\end{para}

\begin{prop}\label{retraction}
	Let $L$ be a conilpotent dg Lie coalgebra, let $U^c L$ be its universal enveloping conilpotent coalgebra, and let $A$ a commutative dg algebra. We consider the map of complete filtered dg Lie algebras
	\[
	\hom_\K(L,A) \longrightarrow \hom_\K(U^c L,A)
	\]
	induced by the projection $U^c L \to L$, where both sides are filtered by the respective coradical filtrations. This map admits a filtration-preserving left inverse which is a morphism of $\hom_\K(L,A)$-modules.
\end{prop}

\begin{proof}
	The dual PBW theorem in the form \eqref{infcobimod} shows that the projection
	\[
	\coass\longrightarrow\colie 
	\]
	admits a right inverse 
	\[\colie \to \coass
	\] as a morphism of infinitesimal $\colie$-bicomodules. Such a right inverse induces a map
	\[ 
	L = \colie\circ^\colie L \longrightarrow  \coass\circ^\colie L = U^c L ,
	\]
	which is right inverse to the projection $U^c L \to L$. 
	The discussion of \S\ref{comodule functoriality} shows that this is a morphism of $L$-comodules. Applying $\hom_\K(-,A)$ gives a map
	\[
	\hom_\K(U^cL, A)\longrightarrow\hom_\K(L, A),
	\]
	which is left inverse to $\hom_\K(L,A) \longrightarrow \hom_\K(U^c L,A)$, and it is straightforward to see from \S\ref{lie functoriality} that it is a morphism of  $\hom_\K(L,A)$-modules.  The discussion of \S\ref{filtration and conilpotence} shows that this retraction is compatible with the filtrations, where the convolution algebras inherit filtrations from the coradical filtrations of the coalgebras.
\end{proof}

\section{Models for mapping spaces of algebras}\label{section: mapping spaces}

\begin{para}\label{hinich model structure}
	Hinich \cite{hinichmodelstructure} has constructed a simplicial model structure on unbounded dg algebras over an operad $\PP$ over a field $\K$ of characteristic zero. This model structure is right transferred from the model structure on unbounded chain complexes along the forgetful functor. The resulting weak equivalences are the morphisms of $\PP$-algebras that are quasi-isomorphisms of chain complexes, and the fibrations are the morphisms that are degree-wise surjections. Any triangulated algebra \cite[Section B.6.7]{lodayvallette} is cofibrant. In particular, if $\PP$ is a Koszul operad, then a canonical cofibrant replacement functor is given by the bar-cobar resolution. The simplicial enrichment is given by
	\[
	\map{\PP}(A, A')\coloneqq\hom_{\PP\text{-}\mathsf{alg}}(A, A'\otimes\Omega_\bullet)\in\mathsf{sSet},
	\]
	where $\Omega_\bullet$ is Sullivan's algebra of polynomial differential forms on the simplices.
\end{para}

\begin{para}
	For a complete filtered dg Lie algebra $\g$, Hinich \cite{hinich97} has also introduced the \emph{Maurer--Cartan simplicial set}
	\[
	\MC_\bullet(\g)\coloneqq\MC(\g \cotimes \Omega_\bullet),
	\]
	where $\g \cotimes \Omega_\bullet \coloneqq \varprojlim( \g/F^n \g \otimes \Omega_\bullet) $. It is a non-trivial result that this simplicial set is in fact a Kan complex. Note in particular that $\MC_0(\g)=\MC(\g)$. Since $\g$ has a complete filtration, we may also exponentiate $\g_0$ to get a \emph{gauge group} acting on $\MC(\g)$, and two Maurer--Cartan elements of $\g$ lie in the same connected component of $\MC_\bullet(\g)$ if and only if they are gauge equivalent. 
\end{para}

\begin{para}
	We can now prove \cref{mappingspaces} from the introduction.
\end{para}

\begin{prop}[\ref{mappingspaces}]\label{thm:mapping space}
	Let $A$ and $A'$ be dg algebras over a field $\K$ of characteristic zero. Then:
	\begin{itemize}
		\item[$\diamond$] $\rmap{\dga}(A, A')\simeq\MC_\bullet(\hom_\K(\Bass A, A'))$.
		\item[$\diamond$] If $A$ and $A'$ are commutative, then $\rmap{\cdga}(A, A')\simeq\MC_\bullet(\hom_\K(\Bcom A, A'))$.
	\end{itemize}
	Here, both $\hom_\K(\Bass A, A')$ and $\hom_\K(\Bcom A, A')$ are considered as complete filtered dg Lie algebras, with the filtrations induced by the coradical filtrations of $\Bass A$ and $\Bcom A$ respectively. In the commutative case, the projection $\Bass A\to \Bcom A$ induces the natural map $\rmap{\dga}(A, A')\longrightarrow\rmap{\cdga}(A, A')$.
\end{prop}

\begin{proof}
	Both assertions have essentially the same proof, so we will present only the demonstration of the first one. The simplicial set $\rmap{\dga}(A, A')$ depends on a choice of cofibrant replacement functor. We choose the bar-cobar resolution $\Cobass\Bass$. In simplicial degree $n$, the left hand side is given by
	\[
	\rmap{\dga}(A, A')_n = \hom_{\dga}\left(\Cobass\Bass A, A'\otimes\Omega_n\right),
	\]
	which by \cite[Theorem~11.3.1]{lodayvallette} is the set of Maurer--Cartan elements of the dg Lie algebra
	\[
	\hom_\K\left(\Bass A,A'\otimes\Omega_n\right).
	\]
	Thus, we need to compare the two simplicial sets
	\[\MC\left(\hom_\K\left(\Bass A,A'\right) \cotimes \Omega_\bullet\right) \qquad\text{and}\qquad
	\MC\left(\hom_\K\left(\Bass A,A'\otimes\Omega_\bullet\right)\right).
	\]The  natural inclusion $\hom_\K\left(\Bass A,A'\right)\otimes\Omega_\bullet\xhookrightarrow{\quad}\hom_\K\left(\Bass A,A'\otimes\Omega_\bullet\right)$ is a morphism of filtered modules when both sides are filtered by the coradical filtration of $\Bass A$. Since the filtration on the right hand side is complete, this map factors canonically through the completion, giving a map
	\[
	\hom_\K\left(\Bass A,A'\right)\cotimes\Omega_\bullet\xhookrightarrow{\quad}\hom_\K\left(\Bass A,A'\otimes\Omega_\bullet\right)
	\]
	which in general is not an isomorphism. Nevertheless, it induces a homotopy equivalence of Maurer--Cartan spaces by \cite[Theorem 4.1]{rnw19}.
	\end{proof}\begin{rem}We remark that \cite[Theorem 4.1]{rnw19} does not explicitly consider the completed tensor product, but the completed tensor product is implicitly used in the arguments. We will discuss the proof of \cite[Theorem 4.1]{rnw19} in more detail in \S\S\ref{rnw summary start}--\ref{rnw summary end}.
\end{rem}

\begin{rem}
	The model for mapping spaces constructed in \cref{thm:mapping space} is evidently functorial in both $A$ and $A'$, defining a bifunctor from $\dga^{\mathrm{op}} \times \dga$ (resp.\ $\cdga^{\mathrm{op}} \times \cdga$) to Kan complexes. In fact more is true: it is also functorial with respect to $\Aoo$-morphisms, resp.\ $\Coo$-morphisms, in both arguments. Indeed, it is shown in \cite{rnw-convolution} that the convolution algebra construction $\hom_\K(-,-)$ is functorial with respect to $\infty$-morphisms in either of its two arguments, but not simultaneously. But an $\infty$-morphism of algebras induces a strict morphism between bar constructions, so  $\hom_\K(\Bass A, A')$ is in fact simultaneously functorial with respect to $\infty$-morphisms in both $A$ and $A'$. It straightforward to check that the quasi-isomorphisms of \cref{mappingspaces} are natural in both slots.
\end{rem}

\begin{rem}
	The same model
	\[
	\rmap{\PP}(A, A')\simeq\MC_\bullet\left(\hom_\K(\Bar_\PP A, A')\right)
	\]
	works for algebras over any Koszul binary operad, by the same argument. It also works outside of the binary case, with the difference that the convolution algebra is an $\Loo$-algebra in that case, instead of a dg Lie algebra \cite[Section 7]{WierstraAlgHop}.
\end{rem}

\begin{para}
	We can now prove the non-unital case of \cref{mainthm2} from the introduction.
\end{para}

\begin{thm}[\ref{mainthm2}]
	Over a base field of characteristic zero, the natural map
	\[
	\rmap{\cdga}(A, A')\longrightarrow\rmap{\dga}(A, A')
	\]
	induced by the forgetful functor induces an injection on all homotopy groups for any $A, A'\in\cdga$ and any basepoint.
\end{thm}

\begin{proof}
	By \cref{thm:mapping space} we can model $\rmap{\cdga}(A,A')\to\rmap{\dga}(A,A')$ by applying $\MC_\bullet$ to the morphism of complete filtered dg Lie algebras $\hom_\K(\Bcom A,A') \to \hom_\K(\Bass A,A')$. By \cref{criterion} we are done if we can construct a filtered retraction from $\hom_\K(\Bass A,A')$ onto $\hom_\K(\Bcom A,A') $ as filtered $\hom_\K(\Bcom A,A') $-modules. This is exactly what is given by \cref{retraction}, since $\Bass A$ is canonically isomorphic to $ U^c \Bcom A$ (this is well known, and can be proven by dualizing the argument of \cite[Lemma 4.7]{CPRNW19}).
\end{proof}

\section{The unital case}\label{sect unital}

\begin{para}
	In this section, we show the unital version of Theorem \ref{mainthm}. Recall that we denote by $\udga$ (resp.\ $\ucdga$) the categories of unital associative (resp.\ commutative) dg algebras over our fixed base field $\K$.
\end{para}

\begin{thm}\label{unitalmainthm}
	Over a base field of characteristic zero, the forgetful functor
	\[
	\ho{\ucdga}\longrightarrow\ho{\udga}
	\]
	is faithful. More generally,  the map $
	\rmap{\ucdga}(A, A') \to\rmap{\udga}(A, A')
	$
	induced by the forgetful functor induces an injection on all homotopy groups, for any $A, A'\in\ucdga$.
\end{thm}

\begin{para}
	The main difference in the unital case is the following. Our strategy for proving \cref{mainthm} was to model the two mapping spaces as Maurer--Cartan simplicial sets of two dg Lie algebras, and then to apply our criterion \cref{criterion}. However, in the categories of unital (commutative) dg algebras, it is simply not possible that mapping spaces can be modeled by dg Lie algebras in this way. Indeed, dg Lie algebras intrinsically model \emph{pointed} spaces, since the element $0$ provides a canonical basepoint of the Maurer--Cartan space. The space of maps between two non-unital rings is canonically pointed by the zero map, but the space of maps between two unital rings has no natural basepoint. In fact the space of maps may even be empty: for example, there does not exist any homomorphism $\mathbf C \to \mathbf R$, derived or otherwise. To overcome this we need to systematically replace dg Lie algebras with \emph{curved Lie algebras}. 
\end{para}

\begin{para}
	Indeed, from the point of view of rational homotopy theory, the choice of an augmentation of a commutative dg algebra corresponds to fixing a basepoint of a space. Thus, one may think that unital algebras without augmentations should correspond to unbased spaces. It is a remarkable insight of Positselski \cite{positselski} that Koszul duality can be extended to unital algebras by introducing a \emph{curvature} term on the other side of the duality: the Koszul dual of a unital commutative dg algebra is a \emph{curved} Lie coalgebra. Similarly counital coalgebras are Koszul dual to curved algebras. It is then natural that Maurer--Cartan spaces of curved Lie algebras can be fruitfully used to model unbased spaces. Further operadic and homotopical studies of curved Lie (or $\Loo$-) algebras have been pursued by  \cite{hirshmilles,braun2012operads,maunder,LeGrignou,le2018homotopy,calaque2021lie,RocaLucio}. 
\end{para}

\begin{para}
	The proof of \cref{unitalmainthm} will closely follow the proof of \cref{mainthm}, except that we need to replace \cref{criterion} and \cref{mappingspaces} with the following curved counterparts.
\end{para}

\begin{prop}\label{curved criterion}
	Let $i \colon \h \to \g$ be a morphism of complete filtered curved Lie algebras over a field of characteristic zero. Suppose that there exists $r \colon \g \to \h$ such that $r \circ i = \mathrm{id}_\h$, and $r$ is a morphism of filtered $\h$-modules. Then the map $\MC_\bullet(\h) \longrightarrow \MC_\bullet(\g)$ induces an injection on all homotopy groups for all basepoints.
\end{prop}

\begin{prop}\label{curved mapping spaces}
	Let $A$ and $A'$ be unital dg algebras over a field $\K$ of characteristic zero. Then:
	\begin{itemize}
		\item[$\diamond$] $\rmap{\udga}(A, A')\simeq\MC_\bullet(\hom_\K(\Buass A, A'))$.
		\item[$\diamond$] If $A$ and $A'$ are commutative, then $\rmap{\ucdga}(A, A')\simeq\MC_\bullet(\hom_\K(\Bucom A, A'))$.
	\end{itemize}
	Here, both $\hom_\K(\Buass A, A')$ and $\hom_\K(\Bucom A, A')$ are considered as complete filtered curved Lie algebras, with the filtrations induced by the coradical filtrations of $\Buass A$ and $\Bucom A$ respectively. In the commutative case, the projection $\Buass A\to \Bucom A$ induces the natural map $\rmap{\udga}(A, A')\longrightarrow\rmap{\ucdga}(A, A').$
\end{prop}

\begin{para}
	Once \cref{curved criterion} and \cref{curved mapping spaces} are in place, the proof proceeds just as in the uncurved case, by observing that the Poincar\'e--Birkhoff--Witt theorem implies the existence of a retraction $\Bucom A \to \Buass A$. In the rest of this section we explain why \cref{curved criterion} and \cref{curved mapping spaces} are true. \end{para}

\begin{para}\label{twisting 1}
	The proof of \cref{curved criterion} formally reduces to its uncurved analogue, \cref{criterion}, by means of \emph{twisting}. Recall that if $\g$ is a complete filtered dg Lie algebra and $\alpha \in \MC(\g)$, then we may twist the differential of $\g$ by $\alpha$ to obtain a new complete filtered dg Lie algebra $\g^\alpha$. If we think of complete filtered dg Lie algebras as models for pointed spaces, then this construction can be understood geometrically as changing the basepoint: given a pointed space $(X,x_0)$ and a point $x_1 \in X$, we obtain a new pointed space $(X,x_1)$. If $\g$ is instead a complete filtered \emph{curved} Lie algebra, then we may similarly twist $\g$ by an \emph{arbitrary} element $\alpha \in \g_{-1}$, obtaining a new complete filtered curved Lie algebra $\g^\alpha$. See for instance \cite[Section 7.1]{braun2012operads}. In the special case that $\alpha$ is a Maurer--Cartan element, the twisted algebra $\g^\alpha$ has vanishing curvature term; that is, $\g^\alpha$ is a dg Lie algebra when $\alpha \in \MC(\g)$. Recalling that complete filtered  curved Lie algebras model unbased spaces, the geometric interpretation is instead that if $X$ is an unbased space and $x_1 \in X$ is a point, then $(X,x_1)$ is a pointed space. 
\end{para}

\begin{para}\label{twisting 2}
	It is now clear how \cref{curved criterion} follows from \cref{criterion}. Indeed, suppose first in the situation of \cref{curved criterion} that $\MC(\h)=\varnothing$. In this case the statement is vacuously true. Otherwise, pick a Maurer--Cartan element $\alpha \in \h$, and apply \cref{criterion} to the morphism of dg Lie algebras $\h^\alpha \to \g^{i(\alpha)}$. 
\end{para}

\begin{para}
	The proof of \cref{curved mapping spaces} follows closely the proof of \cref{thm:mapping space}.  We need three ingredients in order to carry out the argument: 
	\begin{itemize}
		\item [$\diamond$] The bar-cobar construction $\Cobucom \Bucom A \to A$ is a cofibrant replacement functor in the category $\ucdga$.
		\item [$\diamond$] If $A$ is a unital cdga and $L$ is a curved conilpotent Lie coalgebra, then Maurer--Cartan elements of the curved Lie algebra $\hom_\K(L,A)$ naturally correspond to unital cdga morphisms $\Cobucom L \to A$.
		\item [$\diamond$] The natural inclusion
		$ \hom_\K(\Bucom A,A') \cotimes \Omega_\bullet \hookrightarrow \hom_\K(\Bucom A,A' \otimes \Omega_\bullet) $
		induces a homotopy equivalence on simplicial Maurer--Cartan spaces. 
	\end{itemize}
	(And, of course, we also need the evident analogues of the above statements in the associative case.) The fact that $\Cobucom \Bucom A \to A$ is a quasi-isomorphism is part of curved Koszul duality \cite[Proposition 5.2.8]{hirshmilles}). Moreover, the algebra $\Cobucom \Bucom A$ is quasi-free, and its space of generators (i.e.\ $\Bucom A$) is equipped with a ``good'' filtration, namely the filtration induced by the weight grading of $\mathsf{uCom}^\antishriek$. Hence the bar-cobar resolution is cofibrant, which shows the first bullet point. The second bullet point is \cite[Proposition 134]{LeGrignou}. Thus, all that remains to explain is the last bullet point. The argument is a minor modification of the one used to prove \cite[Theorem 4.1]{rnw19}. We now recall this argument and then explain how it should be modified. 
\end{para}

\begin{para}\label{rnw summary start}
	Let us first review a construction due to Getzler \cite{getzler09mcgroupoid}. Getzler considers the \emph{Dupont contraction}, which is an explicit contraction
	\[
	\hbox{
		\begin{tikzpicture}
			\def\upshift{0.075}
			\def\downshift{0.075}
			\pgfmathsetmacro{\midshift}{0.005}
			
			\node[left] (x) at (0, 0) {$\Omega_\bullet$};
			\node[right=1.5 cm of x] (y) {$C_\bullet$};
			
			\draw[->] ($(x.east) + (0.1, \upshift)$) -- node[above]{\mbox{\tiny{$p_\bullet$}}} ($(y.west) + (-0.1, \upshift)$);
			\draw[->] ($(y.west) + (-0.1, -\downshift)$) -- node[below]{\mbox{\tiny{$i_\bullet$}}} ($(x.east) + (0.1, -\downshift)$);
			
			\draw[->] ($(x.south west) + (0, 0.1)$) to [out=-160,in=160,looseness=5] node[left]{\mbox{\tiny{$s_\bullet$}}} ($(x.north west) - (0, 0.1)$);
			
	\end{tikzpicture}}
	\]
	in the category of simplicial chain complexes, where $C_n \subset \Omega_n$ is a subcomplex isomorphic to the cellular cochains on $\Delta_n$.  For a complete filtered dg Lie algebra $\g$, we denote by  $\MC_\bullet(\g)^{s_\bullet=0}$
	the subspace of $\MC_\bullet(\g)$ consisting of Maurer--Cartan elements which satisfy the ``gauge condition'' $s_\bullet(\alpha)=0$.\footnote{Getzler denotes this simplicial set $\gamma_\bullet(\g)$.} Getzler shows for nilpotent $\g$ that $\MC_\bullet(\g)^{s_\bullet=0}$ is a Kan complex and that $\MC_\bullet(\g)^{s_\bullet=0} \hookrightarrow \MC_\bullet(\g)$ is a homotopy equivalence. The result extends to the complete filtered case as an easy consequence of Getzler's arguments, see \cite[Proposition 5.4]{berglund2015rational}.
\end{para}

\begin{para}
	Returning to the inclusion of simplicial dg Lie algebras 
	\[
	\hom_\K(\Bcom A,A') \cotimes \Omega_\bullet \xhookrightarrow{\quad} \hom_\K(\Bcom A,A' \otimes \Omega_\bullet),
	\]
	the point is now that the Dupont contraction operates on both sides of the inclusion. Thus we obtain a commuting diagram of simplicial sets
	\[
	\hbox{
		\begin{tikzpicture}
			\def\xdist{1};
			\def\ydist{0.8};
			
			\node (a) at (0, 0) {$\MC(\hom_\K(\Bcom A,A')\cotimes\Omega_\bullet)^{s_\bullet=0}$};
			\node[right] (b) at ($(a.east) + (\xdist, 0)$) {$\MC(\hom_\K(\Bcom A,A') \cotimes \Omega_\bullet)$};
			\node[below] (c) at ($(a.south) + (0, -\ydist)$) {$\MC(\hom_\K(\Bcom A,A' \otimes \Omega_\bullet))^{s_\bullet=0}$};
			\node[below] (d) at ($(b.south) + (0, -\ydist)$) {$\MC(\hom_\K(\Bcom A,A' \otimes \Omega_\bullet)).$};
			
			\draw[->] (a) to (b);
			\draw[->] (a) to (c);
			\draw[->] (b) to (d);
			\draw[->] (c) to (d);
		\end{tikzpicture}
	}
	\]
	What Getzler showed is precisely that the top horizontal arrow is an equivalence. Repeating his argument with obvious modifications shows that the bottom horizontal arrow is an equivalence, too. The right vertical arrow is the one we want to prove is an equivalence. The left vertical arrow will turn out to be an \emph{isomorphism} of simplicial sets, finishing the argument.
\end{para}

\begin{para}\label{rnw summary end}
	To see that 
	\[
	\MC(\hom_\K(\Bcom A,A') \cotimes \Omega_n)^{s_n=0} \longrightarrow \MC(\hom_\K(\Bcom A,A' \otimes \Omega_n))^{s_n=0}
	\]
	is a bijection for all $n$, we use a reinterpretation of Getzler's gauge condition due to Bandiera \cite{bandiera2017descent}.
	Suppose that we are given a complete filtered $\Loo$- algebra $\mathfrak G$, and a contraction
	\[
	\hbox{
		\begin{tikzpicture}
			\def\upshift{0.075}
			\def\downshift{0.075}
			\pgfmathsetmacro{\midshift}{0.005}
			
			\node[left] (x) at (0, 0) {$\mathfrak{G}$};
			\node[right=1.5 cm of x] (y) {$\mathfrak{H}$};
			
			\draw[->] ($(x.east) + (0.1, \upshift)$) -- node[above]{\mbox{\tiny{$p$}}} ($(y.west) + (-0.1, \upshift)$);
			\draw[->] ($(y.west) + (-0.1, -\downshift)$) -- node[below]{\mbox{\tiny{$i$}}} ($(x.east) + (0.1, -\downshift)$);
			
			\draw[->] ($(x.south west) + (0, 0.1)$) to [out=-160,in=160,looseness=5] node[left]{\mbox{\tiny{$s$}}} ($(x.north west) - (0, 0.1)$);
			
	\end{tikzpicture}}
	\]
	of filtered chain complexes. To this data, the procedure of Homotopy Transfer \cite[Section 10.3]{lodayvallette} associates a canonical structure of complete filtered $\Loo$-algebra on $\mathfrak H$, and Bandiera shows that $\MC(\mathfrak G)^{s=0} = \MC(\mathfrak H)$. In our situation, this means that the simplicial dg Lie algebra structure on $\g \cotimes \Omega_\bullet$ induces a simplicial $\Loo$-structure on $\g \otimes C_\bullet = \g \cotimes C_\bullet$,\footnote{We have $\g \otimes C_\bullet = \g \cotimes C_\bullet$ since $C_n$ is a finite-dimensional chain complex for all $n$.} and that
	\[
	\MC_\bullet(\g)^{s_\bullet=0} = \MC(\g \otimes C_\bullet).
	\]
	But this finishes the result: we are reduced to showing that 
	\[
	\MC(\hom_\K(\Bcom A,A') \otimes C_n) \to \MC(\hom_\K(\Bcom A,A' \otimes C_n))
	\]
	is a bijection. But $\hom_\K(\Bcom A,A') \otimes C_n \to \hom_\K(\Bcom A,A' \otimes C_n)$ is an isomorphism, since $C_n$ is a finite-dimensional chain complex. 
	\end{para}

\begin{rem}
	One can also show directly that $\MC(\g \otimes C_\bullet) \to \MC(\g \cotimes \Omega_\bullet)$ is a homotopy equivalence, without passing through Getzler's gauge condition, although ultimately the arguments are quite similar \cite[Theorem 3.3]{RN17}.
\end{rem}

\begin{para}
	Let us now return to the unital case, and to the inclusion of simplicial curved Lie algebras
	\[
	\hom_\K(\Bucom A,A') \cotimes \Omega_\bullet \xhookrightarrow{\quad} \hom_\K(\Bucom A,A' \otimes \Omega_\bullet).
	\]
	Again, we get a commuting diagram
	\[
	\hbox{
		\begin{tikzpicture}
			\def\xdist{1};
			\def\ydist{0.8};
			
			\node (a) at (0, 0) {$\MC(\hom_\K(\Bucom A,A') \cotimes \Omega_\bullet)^{s_\bullet=0}$};
			\node[right] (b) at ($(a.east) + (\xdist, 0)$) {$\MC(\hom_\K(\Bucom A,A') \cotimes \Omega_\bullet)$};
			\node[below] (c) at ($(a.south) + (0, -\ydist)$) {$\MC(\hom_\K(\Bucom A,A' \otimes \Omega_\bullet))^{s_\bullet=0}$};
			\node[below] (d) at ($(b.south) + (0, -\ydist)$) {$\MC(\hom_\K(\Bucom A,A' \otimes \Omega_\bullet)).$};
			
			\draw[->] (a) to (b);
			\draw[->] (a) to (c);
			\draw[->] (b) to (d);
			\draw[->] (c) to (d);
		\end{tikzpicture}
	}
	\]
	The horizontal arrows are still equivalences: indeed, this reduces to the uncurved case by ``twisting'' as in \S\S\ref{twisting 1}--\ref{twisting 2}. 	The left vertical arrow is still an isomorphism: the interpretation of the gauge condition in terms of Homotopy Transfer is proven in the curved setting in \cite{getzler2018maurer}. 
	To be precise, Getzler shows that in the curved setting it also holds $\MC(\mathfrak G)^{s=0} = \MC(\mathfrak H)$, where $\mathfrak H$ inherits from homotopy transfer a curved $\Loo$-algebra structure. As before, the conclusion follows from
	\begin{align*}
		\MC(\hom_\K(\Bucom A,A') \cotimes \Omega_\bullet)^{s_\bullet=0}&{} = \MC(\hom_\K(\Bucom A,A') \otimes C_\bullet) \stackrel{\cong}{\longrightarrow} \\
		&{}\stackrel{\cong}{\longrightarrow}\MC(\hom_\K(\Bucom A,A' \otimes C_\bullet)) = 	\MC(\hom_\K(\Bucom A,A' \otimes \Omega_\bullet))^{s_\bullet=0}.
	\end{align*}
\end{para}

\section{Enveloping algebras}\label{sec:enveloping algebras}

\begin{para}
	In this section we prove \cref{thm:B}. We write $\Blie$ for the bar construction from dg Lie algebras to dg cocommutative coalgebras and $\Coblie$ for its left adjoint cobar construction.
\end{para}

\begin{lem}\label{enveloping lemma}
	Let $\g$ and $\h$ be dg Lie algebras over a field of characteristic zero. There is a commutative diagram
	\[
	\hbox{
			\begin{tikzpicture}
					\def\xdist{2};
					\def\ydist{0.8};
					
					\node (a) at (0, 0) {$\MC_\bullet(\hom_\K(\Blie \g, \h))$};
					\node[right] (b) at ($(a.east) + (\xdist, 0)$) {$\MC_\bullet(\hom_\K(\Blie \g, U\h)).$};
					\node[below] (c) at ($(a.south) + (0, -\ydist)$) {$\rmap{\dgla}(\g,\h)$};
					\node[below] (d) at ($(b.south) + (0, -\ydist)$) {$\rmap{\dga}(U\g,U\h).$};
					
					\draw[->] (a) to (b);
					\draw[<-] (c) to node[sloped, below=-0.05cm]{$\sim$} (a);
					\draw[<-] (d) to node[sloped, below=-0.05cm]{$\sim$} (b);
					\draw[->] (c) to (d);
				\end{tikzpicture}
		}
	\]
	in which the vertical arrows are homotopy equivalences, the upper horizontal arrow is induced by the natural map $\h \to U\h$, and the lower horizontal arrow is induced from the universal enveloping algebra functor. 
\end{lem}

\begin{proof}
	We remark that the left vertical arrow is constructed by the exact same arguments as in \cref{thm:mapping space}. As in the proof of \cref{thm:mapping space}, the two simplicial mapping spaces depend on the choice of a cofibrant replacement of $\g$ and $U\g$, respectively. For $\g$ we choose the bar-cobar resolution $\Coblie\Blie \g \stackrel\sim\to \g$. But for $U\g$ we choose instead to take the resolution
	\[
	\Cobass \Blie \g \stackrel\sim\longrightarrow U\g.
	\]
	This map is constructed as follows: the universal enveloping algebra functor preserves quasi-isomorphisms by the PBW theorem, so from the bar-cobar resolution of $\g$ we get a quasi-isomorphism $U\Coblie \Blie\g \stackrel \sim \to U\g$. But if $C$ is a cocommutative coalgebra, then $\Omega C \cong U \Coblie C$ \cite[Lemma 4.7]{CPRNW19}, so $  U \Coblie\Blie \g \cong \Cobass \Blie \g $. The algebra $\Omega \Blie \g$ is triangulated, hence cofibrant, by the same argument that shows that any bar-cobar resolution is cofibrant. 
	
	With these choices of cofibrant replacement, the simplicial enrichment constructed by Hinich (\S\ref{hinich model structure}) gives 
	\[
	\rmap{\dgla}(\g,\h)_n = \mathrm{Hom}_\dgla(\Coblie\Blie \g,\h \otimes \Omega_n) = \MC(\hom_\K(\Blie \g, \h \otimes \Omega_n))
	\]
	and
	\[
	\rmap{\dga}(U\g,U\h)_n = \mathrm{Hom}_\dga(\Cobass \Blie \g,U\h \otimes \Omega_n) = \MC(\hom_\K(\Blie \g, U\h \otimes \Omega_n))
	\]
	with the map between them given by the universal enveloping algebra functor. But now we have in each simplicial degree a commuting diagram
	\[
	\hbox{
		\begin{tikzpicture}
			\def\xdist{1};
			\def\ydist{0.8};
			
			\node (a) at (0, 0) {$\MC(\hom_\K(\Blie \g, \h) \cotimes \Omega_n)$};
			\node[right] (b) at ($(a.east) + (\xdist, 0)$) {$\MC(\hom_\K(\Blie \g, U\h)  \cotimes \Omega_n)$};
			\node[below] (c) at ($(a.south) + (0, -\ydist)$) {$\MC(\hom_\K(\Blie \g,\h \otimes \Omega_n))$};
			\node[below] (d) at ($(b.south) + (0, -\ydist)$) {$\MC(\hom_\K(\Blie \g,U\h \otimes \Omega_n))$};
			
			\draw[->] (a) to (b);
			\draw[->] (a) to (c);
			\draw[->] (b) to (d);
			\draw[->] (c) to (d);
		\end{tikzpicture}
	}
	\]
	with the horizontal arrows induced by $\h \to U\h$, which produces the commuting square of the lemma. The vertical arrows realize to homotopy equivalences by \cite[Theorem 4.1]{rnw19}.
\end{proof}

\begin{proof}[Proof of \cref{thm:B}]
	Applying \cref{enveloping lemma} and \cref{criterion}, it suffices to show that the map \[\hom_\K(\Blie \g, \h)  \to \hom_\K(\Blie \g, U\h)\] induced by the embedding $\h \to U\h$ admits a filtration-preserving left inverse as a map of $\hom_\K(\Blie \g, \h)$-modules. But $\h \to U\h$ admits a left inverse as a morphism of $\h$-modules, by the usual PBW theorem, and this gives the retraction we want by the functoriality explained in \S\ref{lie functoriality}. It clearly respects the filtrations, since the relevant filtrations are induced by the coradical filtration on $\Blie \g$. 
\end{proof}

\bibliographystyle{alpha}
\bibliography{bib}

\end{document}